\documentclass[12pt,a4paper]{amsart}

\usepackage{mathrsfs}

\newtheorem{theo+}              {Theorem}           [section]
\newtheorem{prop+}  [theo+]     {Proposition}
\newtheorem{coro+}  [theo+]     {Corollary}
\newtheorem{lemm+}  [theo+]     {Lemma}
\newtheorem{exam+}  [theo+]     {Example}
\newtheorem{rema+}  [theo+]     {Remark}
\newtheorem{defi+}  [theo+]     {Definition}
\newtheorem{clai+}  [theo+]     {Claim}
\newenvironment{theorem}{\begin{theo+}}{\end{theo+}}
\newenvironment{proposition}{\begin{prop+}}{\end{prop+}}
\newenvironment{corollary}{\begin{coro+}}{\end{coro+}}
\newenvironment{lemma}{\begin{lemm+}}{\end{lemm+}}

\usepackage{amsthm}
\theoremstyle{plain} \theoremstyle{remark}
\newtheorem{remark}{Remark}

\def \r{\mbox{${\mathbb R}$}}
\def\E{/\kern-1.0em \equiv }

\title{  Biharmonic CMC surfaces, harmonic and  biharmonic Riemannian surmersions on
  Berger 3-sphere $S^3_\varepsilon$}
\author{ Ze-Ping Wang$^{*}$ and Ye-Lin Ou$^{**}$}

\address{Department of Mathematics,\newline\indent Guizhou
Normal University,\newline\indent Guiyang 550025,\newline\indent
People's Republic of China
\newline\indent E-mail:zpwzpw2012@126.com \;(Wang)
\\\newline\indent  \\\newline\indent
Department of Mathematics,\newline\indent Texas A $\&$ M University-Commerce,
\newline\indent Commerce TX 75429,\newline\indent USA.\newline\indent
E-mail:yelin$\_$ou@tamu-commerce.edu \;(Ou).}

\thanks{*Supported by the Natural Science Foundation of China (No. 11861022). \\
\indent** Supported by a grant from the Simons Foundation ( 427231,
Ye-Lin Ou).}
\begin{document}
\title[Biharmonic isometric immersions and Riemannian submersions ] {Biharmonic isometric immersions into and  biharmonic Riemannian submersions  from Berger 3-spheres}
\date {18/2/2023}  \maketitle

\section*{Abstract}
\begin{quote}
{\footnotesize In this paper,  we study  biharmonic isometric immersions of a surface  into and  biharmonic Riemannian submersions
 from  3-dimensional Berger spheres.
We obtain a classification of proper biharmonic isometric immersions of a surface with constant mean curvature into
Berger 3-spheres. We also give a complete classification of proper biharmonic
Hopf tori in  Berger 3-sphere.  For Riemannian submersions, we  prove that a Riemannian submersion from  Berger 3-spheres into a surface is biharmonic if and only if it is harmonic.}
\end{quote}
\section{Introduction and preliminaries}

In this paper, we work in the category of smooth objects, so
manifolds, maps, vector fields, etc, are assumed to be smooth unless it is stated otherwise.\\

Recall a harmonic map $\varphi:(M, g)\to (N,
h)$ of a compact Riemannian manifold
$(M, g)$ into another Riemannian manifold $(N, h)$  that if $\varphi|_{\Omega}$ is  a critical point
of the energy functional defined by
\begin{equation}\nonumber
E\left(\varphi,\Omega \right)= \frac{1}{2} {\int}_{\Omega}
\left|{\rm d}\varphi \right|^{2}{\rm d}x.
\end{equation}
The Euler-Lagrange equation (see \cite{BW1,EL1})  is given by the vanishing of the tension field $\tau(\varphi)={\rm
Trace}_{g}\nabla {\rm d} \varphi$, i.e., $\tau(\varphi)={\rm
Trace}_{g}\nabla {\rm d} \varphi=0.$ \\
 In 1983, J. Eells and L. Lemaire \cite{EL1} extended the notion of harmonic maps to {\em biharmonic maps} which are critical points of the bienergy functional
\begin{equation}\label{bef}\notag
E^{2}\left(\varphi,\Omega \right)= \frac{1}{2} {\int}_{\Omega}
\left|\tau(\varphi) \right|^{2}{\rm d}x,
\end{equation}
for every compact subset $\Omega$ of $M$, where $\tau(\varphi)={\rm
Trace}_{g}\nabla {\rm d} \varphi$ is the tension field of $\varphi$. In 1986, G.Y. Jiang \cite{Ji} first computed the first
variation of the functional, and obtained that $\varphi$
is  biharmonic  if and only if its bitension field vanishes
identically, i.e.,
\begin{equation}\label{BT1}\notag
\tau^{2}(\varphi):={\rm
Trace}_{g}(\nabla^{\varphi}\nabla^{\varphi}-\nabla^{\varphi}_{\nabla^{M}})\tau(\varphi)
- {\rm Trace}_{g} R^{N}({\rm d}\varphi, \tau(\varphi)){\rm d}\varphi
=0,
\end{equation}
where $R^{N}$ is the curvature operator of $(N, h)$ defined by
$$R^{N}(X,Y)Z=
[\nabla^{N}_{X},\nabla^{N}_{Y}]Z-\nabla^{N}_{[X,Y]}Z.$$
Naturally,  any harmonic map  is always biharmonic.\\
 A Riemannian submersion is called a {\bf biharmonic  Riemannian submersion} if the Riemannian submersion is a biharmonic map.
Similarly, a submanifold is called a biharmonic submanifold if the isometric immersion that defines the submanifold is a biharmonic map.
As is well known, an isometric immersion is harmonic if and only if it is minimal, and hence biharmonic submanifolds include minimal submanifolds  as a subset.
We use {\bf proper biharmonic maps (respectively, Riemannian submersion, isometric immersion, submanifold)} to name those biharmonic maps
(respectively, Riemannian submersion, isometric immersion, submanifold) which are not harmonic.\\
 Many recent works in the geometric study of biharmonic maps have been focused on the existence of a proper biharmonic map between  two  ``good'' model spaces. The so-called ``good'' model spaces include space forms,  more general symmetric,  homogeneous spaces, etc. It would be also important to classify all proper biharmonic maps between two model spaces where the
existence is known.  We refer to two classification problems as follows\\
{\bf Chen's conjecture} \cite{CH1, CH2, CH}: every biharmonic submanifold in a Euclidean space $\r^n$ is minimal (i.e., harmonic)\\
{\bf The generalized Chen's conjecture}: every biharmonic submanifold of a Riemannian manifold of non positive curvature must be harmonic (minimal) (see e.g., [4--13]).\\
 The Chen's conjecture is still open for the general case,  and some results for affirmative
answers to  Chen's conjecture were shown in \cite{MO, BMO1, Ou5,  Ou7, FH}. For the generalized Chen's conjecture, Ou and Tang  (\cite{Ou6}) gave
many counter examples in a Riemannian manifold of negative curvature.
For some recent progress on biharmonic submanifolds, we refer the readers to \cite{AO}, [4-13], [23-30], etc., and the
references therein. \\

On the other hand, as it is well known that Riemannian submersions can be considered as the dual notion of  isometric immersions (i.e., submanifolds),  it is very interesting to study biharmonicity of  Riemannian submersions between Riemannian manifolds. In 2002, Oniciuc  \cite{Oni} first studied  biharmonic Riemannian submersions. In 2010,  Wang and Ou \cite{WO} first used the so-called integrability data to study  biharmonicity of a Riemannian submersion from a generic 3-manifold, they then used the main tool to derived a complete classification of biharmonic Riemannian submersions from a 3-dimensional space form into a surface.  In a recent paper \cite{AO},  Akyol and Ou  studied biharmonicity of a general Riemannian submersion and obtained biharmonic equations for Riemannian submersions with one-dimensional fibers and Riemannian submersions with
basic mean curvature vector fields of fibers. In particular, the authors of \cite{AO} used
the so-called integrability data to study biharmonic Riemannian submersions from $(n+1)$-dimensional spaces with one-dimensional fibers and obtained many examples of biharmonic Riemannian submersions. In \cite{Ura2}, the author studied biharmonicity a more general setting of Riemannian submersions with a $S^1$ fiber over a compact Riemannian manifold. In 2018, the authors in \cite{GO}  studied  generalized harmonic morphisms and obtained many examples of biharmonic Riemannian submersions  which are maps between Riemannian manifolds that pull back local harmonic functions to local biharmonic functions.\\

 Finally,  we refer an interested  reader to the recent works  \cite{Ou4} and \cite{WO1} for  complete classifications of constant mean curvature proper biharmonic surfaces in Thurston's 3-dimensional geometries and in BCV 3-spaces, a complete classification of proper biharmonic Hopf cylinders  BCV 3-spaces, complete classification of  proper biharmonic Riemannian submersions  from BCV 3-diemnsional spaces into a surface, and some constructions of examples of proper biharmonic Riemannian submersion from  $H^2\times\r\to \r^2$,\;or,  $\widetilde{SL}(2,\r)\to \r^2$. \\

In this paper,  we will study  biharmonic isometric immersions of a surface  into and  biharmonic Riemannian submersions
from 3-dimensional Berger sphere $S^3_\varepsilon$.
We  show that an isometric immersion of a surface with constant mean curvature into
Berger 3-sphere is proper biharmonc if and only if the surface is  a part of  $S^2(1/\sqrt{2})$ in $S^3$\;or\; a part of a Hopf torus in $S^3_\varepsilon$  whose  base curve is  a circle with radius $r=1/\sqrt{8-4\varepsilon^2}$ in the base sphere $S^2(\frac{1}{2})$. We also give a complete classification of proper biharmonic
Hopf tori in a Berger 3-sphere.  For Riemannian submersions, we  prove that a Riemannian submersion from a Berger 3-sphere into a surface is biharmonic if and only if it is harmonic.

\section{Biharmonic isometric immersions of a surface with constant mean curvature into
Berger 3-sphere $S^3_\varepsilon$}

Biharmonic surfaces in 3-dimensional space forms have been
completely classified in \cite{Ji2}, \cite{CH}, \cite{CMO1}, \cite{CMO2}),  and also  biharmonic  constant mean curvature surfaces in 3-dimensional BCV spaces and Sol space have been
completely classified (\cite{Ou4}). In this section, we  obtain a complete classification of isometric immersions of a surface with constant mean curvature into a
Berger 3-sphere $S^3_\varepsilon$. We also derive a complete classification of proper biharmonic
Hopf turi in a Berger 3-sphere.\\

Let us recall the definition of the so-called 3-dimensional Berger sphere (see e.g.,  \cite{B}).
Consider the Hopf map $ \psi: S^3
(1) \to S^2(\frac{1}{2})$  given by
\begin{equation}\label{cmc-1}
\begin{array}{lll}
\psi(x^1,x^2,x^3,x^4)= \frac{1}{2}
(2x^1x^3 + 2x^2x^4, 2x^2x^3- 2x^1x^4,(x^1)^2 + (x^2)^2-(x^3)^2 -(x^4)^2),
\end{array}
\end{equation}
or
\begin{equation}\label{cmc0}
\begin{array}{lll}
\psi(z, w) = \frac{1}{2}(2zw, |z|^2 - |w|^2) ,
\end{array}
\end{equation}
where $z = x^1 + ix^2$, $w = x^3 + ix^4$ and $S^2(\frac{1}{2})$ denotes a 2-sphere with radius $\frac{1}{2}$ (i.e., constant Gauss curvature $4$ ).
 It is not difficult to see that the map $\psi$ is a Riemannian submersion with
totally geodesic fibers $\psi^{-1}(\psi(z, w))$ which are the great circle passing through $(z, w)$ and
$(iz, iw)$.\\
With respect to the Hopf fibration, the following deformation of the standard metric $g$ on $S^3$
gives a family of metric  on the sphere:
\begin{equation}\label{cmc1}
\begin{array}{lll}
g_\varepsilon|_{ T^HS^3\times T^HS^3} =g|_{ T^HS^3\times T^HS^3} , g_\varepsilon|_{ T^VS^3\times T^VS^3}  = \varepsilon^2g, g_\varepsilon|_{ T^HS^3\times T^VS^3} = 0,
\end{array}
\end{equation}
where $T ^V S^3$
and $T^HS^3$ denote respectively the vertical and the horizontal
spaces determined by $\psi$. We call  a sphere  a Berger
3-sphere if the sphere $S^3$ endowed with the metric $g_\varepsilon$.  A Berger
3-sphere  is denoted by $S^3_\varepsilon$, i.e., $S^3_\varepsilon=(S^3,g_\varepsilon)$, where $\varepsilon\neq0$.
Suppose $x \in S^3$, we have the following facts:\\
(i) the vector fields
\begin{equation}\label{cmc2}
\begin{array}{lll}
X_1(x) = (-x^2, x^1, -x^4, x^3),\;
 X_2(x) = (-x^4, -x^3, x^2, x^1),\\
X_3(x) = (-x^3, x^4, x^1, -x^2)
\end{array}
\end{equation}
parallelize $S^3$,\\
(ii) $X_1$ is tangent to the fibres of the Hopf map (i.e. $d\psi(X_1) = 0$),\;and\\
(iii)$ X_2$ and $X_3$ are horizontal, but not basic.\\

From (\ref{cmc1}) we have a global orthonormal frame field
\begin{equation}\label{cmc3}
\begin{array}{lll}
\{E_1 = X_2, E_2 = X_3, E_3 =\varepsilon^{-1} X_1\}
\end{array}
\end{equation}
 on $S^3_\varepsilon$. \\

We adopt the following notation and sign convention for Riemannian
curvature operator:
\begin{equation}
 R(X,Y)Z=\nabla_{X}\nabla_{Y}Z
-\nabla_{Y}\nabla_{X}Z-\nabla_{[X,Y]}Z,\\
\end{equation}
and the Riemannian and the Ricci curvatures:
\begin{equation}
\begin{array}{lll}
&&  R(X,Y,Z,W)=g( R(Z,W)Y,X),\\
&& {\rm Ric}(X,Y)= {\rm Trace}_{g}R=\sum\limits_{i=1}^3 R(Y, e_i, X,
e_i)=\sum\limits_{i=1}^3 \langle R( X,e_i) e_i, Y\rangle.
\end{array}
\end{equation}
With respect to the frame, a straightforward computation shows that
\begin{equation}\label{Lie}
[E_1,E_2]=2 \varepsilon E_{3},\;\; [E_2,E_3]=\frac{2}{\varepsilon}E_1,\; [E_3,E_1]=\frac{2}{\varepsilon}E_2.
\end{equation}
The Levi-Civita connection
of the metric $g_\varepsilon$  has the expression as
\begin{equation}\label{g1}
\begin{cases}
\nabla_{E_{1}}E_{1}=0,\;\;\nabla_{E_{1}}E_{2}=\varepsilon E_{3},\;\;\nabla_{E_{1}}E_{3}=-\varepsilon E_{2},\\
\nabla_{E_{2}}E_{1}=-\varepsilon E_{3},\;\;\nabla_{E_{2}}E_{2}=0,\;\;\nabla_{E_{2}}E_{3}=\varepsilon E_{1},\\
\nabla_{E_{3}}E_{1}=\frac{2-\varepsilon^2}{\varepsilon}E_{2},\;\;\nabla_{E_{3}}E_{2}=-\frac{2-\varepsilon^2}{\varepsilon} E_{1},\;\;
\nabla_{E_{3}}E_{3}=0.
\end{cases}
\end{equation}

A further computation (see also \cite{B}) gives the possible
nonzero components of the curvatures:
\begin{equation}\label{g2}
\begin{array}{lll}
 R_{1212}=g(R(E_{1},E_{2})E_{2},E_{1})=4-3\varepsilon^2,\\
R_{1313}=g(R(E_{1},E_{3})E_{3},E_{1})=R_{2323}=g(R(E_{2},E_{3})E_{3},E_{2})=\varepsilon^2,\\
{\rm all\;other }\; R_{ijkl}=g(R(E_{k},E_{l})E_{j},E_{i})=0,\;i,j,k,l=1,2,3.
\end{array}
\end{equation}
 and the
Ricci curvature:
\begin{equation}\label{g3}
\begin{array}{lll}
 {\rm Ric}\, (E_{1},E_{1})={\rm Ric}\,
(E_{2},E_{2})=4-2\varepsilon^2,\\{\rm Ric}\,
(E_{3},E_{3})=2\varepsilon^2,\;{\rm all\;other }\; {\rm
Ric}\, (E_i,E_j)=0,\;i\neq j.
\end{array}
\end{equation}

\begin{remark}\label{r0}
 From (i), (ii), (iii), (\ref{cmc3}), (\ref{Lie}) and (\ref{g1}), we would like to point out the following:\\
$(a)$:  The  map $ \psi: S^3_{\varepsilon} \to S^2(\frac{1}{2})$,
$\psi(z, w) = \frac{1}{2}(2zw, |z|^2 - |w|^2)$, where $z = x^1 + ix^2$, $w = x^3 + ix^4$,
is a Riemannian submersion with totally geodesic fibers from a Berger 3-sphere $S^3_{\varepsilon}$  to a 2-sphere $S^2(\frac{1}{2})$  with  constant Gauss curvature $4$,\;i.e., the Riemannian submersion is harmonic.\\
$(b)$:  $\{E_1 = X_2, E_2 = X_3, E_3 =\varepsilon^{-1} X_1\}$ is  an orthonormal  frame on $S^3_{\varepsilon}$  with $E_3$ being vertical. \\
$(c)$: $\{E_1 = X_2, E_2 = X_3\}$ is horizontal, but not basic.\\
\end{remark}

 We will use the following equation for biharmonic hypersurfaces in a
generic Riemannian manifold.

\begin{theorem}$($\cite{Ou3}$)$\label{MTH}
Let $\varphi:M^{m}\longrightarrow N^{m+1}$ be an isometric
immersion of codimension one with mean curvature vector
$\eta=H\xi$. Then $\varphi$ is biharmonic if and only if:
\begin{equation}\label{BHEq}
\begin{cases}
\Delta H-H |A|^{2}+H{\rm
Ric}^N(\xi,\xi)=0,\\
 2A\,({\rm grad}\,H) +\frac{m}{2} {\rm grad}\, H^2
-2\, H \,({\rm Ric}^N\,(\xi))^{\top}=0,
\end{cases}
\end{equation}
where ${\rm Ric}^N : T_qN\longrightarrow T_qN$ denotes the Ricci
operator of the ambient space defined by $\langle {\rm Ric}^N\,
(Z), W\rangle={\rm Ric}^N (Z, W)$ and  $A$ is the shape operator
of the hypersurface with respect to the unit normal vector $\xi$.
\end{theorem}

We now study  biharmonic constant mean curvature (CMC) surfaces in a 3-dimensional
Berger sphere $S^3_\varepsilon$.
\begin{theorem}\label{Th1}
A constant mean curvature surface in  3-dimensional Berger spheres $S^3_\varepsilon$ is proper biharmonic if
and only if it is a part of:\\
$(i)$ $S^2(1/\sqrt{2})$ in $S^3$,\;or\\
$(ii)$  a Hopf torus in $S^3_\varepsilon$, i.e., the inverse image of the Hopf  fibration of a circle of radius $r=\frac{1}{2\sqrt{2-\varepsilon^2}}$ with $\varepsilon^2 < 1$ in the base sphere $S^2(\frac{1}{2})$.
\end{theorem}
\begin{proof}
Let $\{e_1=\sum\limits_{i=1}^3a^i_1E_i,\;\; e_2=\sum\limits_{i=1}^3a^i_2E_i,\;\; \xi=\sum\limits_{i=1}^3a^i_3E_i\}$ be an
orthonormal frame on $S^3_\varepsilon$ adapted to the surface  with $\xi$ being normal. We then use the
Ricci curvature (\ref{g3}) to  have ${\rm
Ric}\,(\xi,\xi)=4-2\varepsilon^2+(4\varepsilon^2-4)(a_3^3)^2, \;\;({\rm
Ric}\,(\xi))^{\top}=(4\varepsilon^2-4)a^3_1a_3^3e_1+(4\varepsilon^2-4)a^3_2a_3^3e_2$.
 From these and the  biharmonic surface  (\ref{BHEq}), we conlude that  a
surface with constant mean curvature $H$ is biharmonic if and only if
\begin{equation}
\begin{cases}
-H \left[|A|^{2}-\left(4-2\varepsilon^2\right)-(4\varepsilon^2-4)(a_3^3)^2\right]=0,\\
(4\varepsilon^2-4)a^3_1a_3^3H=0,\\
(4\varepsilon^2-4)a^3_2a_3^3H=0,
\end{cases}
\end{equation}
which has solution $H=0$ implying that the surface is harmonic (minimal), or,
\begin{equation}\label{th1}
\begin{cases}
|A|^{2}-\left(4-2\varepsilon^2\right)-(4\varepsilon^2-4)(a_3^3)^2=0,\\
(4\varepsilon^2-4)a^3_1a_3^3=0,\\
(4\varepsilon^2-4)a^3_2a_3^3=0.
\end{cases}
\end{equation}
We  solve (\ref{th1}) by considering the following cases:\\
Case I: $\varepsilon=\pm 1$. In this case, we have $|A|^2 = 2$ and the corresponding Berger sphere $S^3_\varepsilon$ is a standard  3-dimensional sphere $S^3$. It follows from \cite{CMO1} \cite{CMO2} that  the only proper biharmonic surface in a 3-dimensional sphere $S^3$ is a part of $S^2(1/\sqrt{2})$ in $S^3$ .\\
Case II: $\varepsilon \ne \pm 1$. In this case, by the last two equations of (\ref{th1}), we have either $a_3^3=0$ or $a_1^3=a_2^3=0$.\\
For Case II-A: $a_3^3=0$, using the first equation of (\ref{th1}) we have
\begin{equation}\label{th2}
|A|^{2}=4-2\varepsilon^2.
\end{equation}
Noting that $a_3^3=0$ implies that the normal vector field of the surface $\Sigma$ is always orthogonal to $E_3 = \varepsilon^{-1}X_1$
so we can choose an another orthonormal frame $\{e_1 = aE_1 + bE_2, e_2 = E_3,\xi = bE_1 -aE_2\}$ adapted to the surface with $a^2 + b^2 = 1$ and $\xi$
being the unit normal vector filed. We use (\ref{g1}) to compute
\begin{equation}\label{th3}
\nabla_{e_1}\xi=\{ae_1(b)-be_1(a)\}e_1-\varepsilon e_2,\;
\nabla_{e_2}\xi=\left(ae_2(b)-be_2(a)+\frac{2-\varepsilon^2}{\varepsilon}\right)e_1.
\end{equation}
With respect to the chosen adapted orthonormal
frame, by a further computation, the second fundamental form of the surface $\Sigma$  given by
\begin{equation}\label{th4}
\begin{array}{lll}
 h(e_1,e_1)=-\langle\nabla_{e_1}\xi,e_1\rangle=-ae_1(b)+be_1(a),\;
 h(e_1,e_2)=-\langle\nabla_{e_1}\xi,e_2\rangle=\varepsilon,\\
 h(e_2,e_1)=-\langle\nabla_{e_2}\xi,e_1\rangle=-ae_2(b)+be_2(a)-\frac{2-\varepsilon^2}{\varepsilon},\;
  h(e_2,e_2)=-\langle\nabla_{e_2}\xi,e_2\rangle=0.
\end{array}
\end{equation}
By (\ref{th3}) and (\ref{th4}), we have
\begin{equation}\label{th5}
\begin{array}{lll}
\nabla_{e_1}e_1=-\{ae_1(b)-be_1(a)\}\xi=2H\xi,\;
\nabla_{e_1}\xi=-2He_1-\varepsilon e_2,\;
\nabla_{e_1}e_2=\varepsilon \xi.
\end{array}
\end{equation}
It follows from (\ref{th4}), the symmetry $h(e_1, e_2) = h(e_2, e_1)$, and $0 = e_2(a^2 + b^2) = 2ae_2(a) + 2be_2(b)$ that $e_2(a) =\frac{2}{\varepsilon}b, e_2(b) = -\frac{2}{\varepsilon}a.$\\
Denoting by $\alpha_1=\xi = bE_1 -aE_2,\alpha_2=e_1 = aE_1 + bE_2$, and $\alpha_3=e_2 = E_3$, a straightforward computation using (\ref{Lie}) and (\ref{th5}) gives
\begin{equation}\label{th7}
\begin{array}{lll}
[\alpha_1,\alpha_3]=0,\;[\alpha_2,\alpha_3]=0,\;

[\alpha_1,\alpha_2]=\left(b\alpha_1(a)-a\alpha_1(b)\right)\alpha_1+2H\alpha_2+2\varepsilon \alpha_3.
\end{array}
\end{equation}
By Remark \ref{r0},  the map
 $\psi: S^3_\varepsilon\to S^2(\frac{1}{2})$
\begin{equation}\label{th8}\notag
\begin{array}{lll}
\psi(x^1,x^2,x^3,x^4)= \frac{1}{2}
(2x^1x^3 + 2x^2x^4, 2x^2x^3- 2x^1x^4,(x^1)^2 + (x^2)^2-(x^3)^2 -(x^4)^2),
\end{array}
\end{equation}
or
\begin{equation}\label{th9}\notag
\begin{array}{lll}
\psi(z, w) = \frac{1}{2}(2zw, |z|^2 - |w|^2) ,
\end{array}
\end{equation}
is a Riemannian submersion with
totally geodesic fibers with an orthonoromal frame $ \{E_1, E_2,\;\alpha_3=E_3\}$
   on $S_{\varepsilon}^3$ with $ E_3$  being  vertical, where $z = x^1 + ix^2$,\;$w = x^3 + ix^4$.

   An interesting thing is that, By (\ref{th7}),  one sees that the orthonormal frame $\{e_1, e_2, \xi\}$ adapted to the constant mean curvature surface happens to be an orthonormal frame $ \{\alpha_1=\xi = bE_1 -aE_2,\alpha_2=e_1 = aE_1 + bE_2,  \alpha_3=e_2 = E_3\} $  adapted to the Hopf fibration $\psi$. The tangent vector field of the surface  $e_2 = E_3$ also implies that the surface contains all the fibers of $\psi$ which intersect the surface. Thus, locally, the surface is formed by the fibers of $\psi$ that pass through every point on an integral curve of $e_1=\alpha_2$ which is horizontal and basic to the Hopf fibration $\psi$. It follows that the constant mean curvature surface is actually a Hopf torus, i.e., the inverse image of a curve on base sphere of the Hopf fibration.

   More precisely, we can determine the torus as follows.

  Let $\gamma:I\to S^3_\varepsilon$, $\gamma=\gamma(s)$, be an  integral curve of the basic vector field  $\alpha_2=e_1 = aE_1 + bE_2$ on the surface $\Sigma$ with arclength parameter, then it is horizontal with respect to the Riemannian submersion $\psi$. Let $\beta(s) = \psi (\gamma(s))$ be the curve in the base space of the Riemannian submersion, then the surface $\Sigma$ is $\Sigma=\cup_{s\in I}\psi^{-1}\left(\beta(s)\right)$, a Hopf torus over the curve $\beta(s)\subset S^2(\frac{1}{2})$.\\

Noting also that $\alpha_1=\xi = bE_1 -aE_2,\alpha_2=e_1 = aE_1 + bE_2, \alpha_3=e_2 = E_3$, then (\ref{th5}) turns into
\begin{equation}\label{th10}
\begin{array}{lll}
\nabla_{\alpha_2}\alpha_2=2H\alpha_1=k_g\alpha_1,\;
\nabla_{\alpha_2}\alpha_1=-2H\alpha_2-\varepsilon \alpha_3=-k_g\alpha_2+\tau_g\alpha_3,\\
\nabla_{\alpha_2}\alpha_3=\varepsilon \alpha_1=-\tau_g\alpha_1,
\end{array}
\end{equation}
which is the Frenet formula of the curve $\gamma=\gamma(s)$ (see also [\cite{Va1},\;Example 3.4.1],  and means that $k_g$ is the geodesic curvature of the base curve,  $\tau_g$ is the geodesic torsion of $\gamma(s)$. It follows from Eqs. (\ref{th4}) and (\ref{th10}) that
\begin{equation}\label{th11}
\begin{array}{lll}
|A|^2=k_g^2+\varepsilon^2,\;
H=k_g/2.
\end{array}
\end{equation}
Comparing (\ref{th2}) and (\ref{th11})) we get
\begin{equation}\label{th12}\notag
\begin{array}{lll}
k_g^2=4(1-\varepsilon^2)>0,\;
H^2=1-\varepsilon^2>0.
\end{array}
\end{equation}
Since the curve in the base sphere $S^2(\frac{1}{2})$ has constant geodesic curvature and hence $k_g=2\sqrt{1-\varepsilon^2}$,  one can check that this curve, considered as a curve in Euclidean 3-space of which $S^2(\frac{1}{2})$ is a subset, has curvature $k=\sqrt{k_g^2+k_n^2}=2\sqrt{2-\varepsilon^2}$ and torsion $\tau=-\frac{2k'}{kk_g}=0$. Combining this, we find the base curve of the Hopf cylinder to be a circle on $S^2(\frac{1}{2})$ with radius $r=\frac{1}{2\sqrt{2-\varepsilon^2}}$.\\
For Case II-B: $a^3_1 = a^3_2 = 0$ and $a_3^3=\pm1$. It follows, in this case,
Span$\{e_1, e_2\}$ = Span$\{E_1, E_2\}$. This implies the distribution determined by $\{E_1, E_2\}$ is integrable and hence (by Frobenius
theorem) is involutive. This leads to $\varepsilon=0$ by (\ref{Lie}), a contradiction.\\
Summarizing all results proved above we obtain the theorem.

\end{proof}

\begin{theorem}\label{Th2}
 Let $\psi: S^3_\varepsilon\to S^2(\frac{1}{2})$, $\psi(z, w) = \frac{1}{2}(2zw, |z|^2 - |w|^2)$
be the Hopf fibration, and $\beta : I \to  S^2(\frac{1}{2})$
be an immersed regular curve parameterized by arc length. Then the Hopf torus $\Sigma=\cup_{s\in I}\psi^{-1}(\beta(s))$ is a proper biharmonic surface in a Berger 3-sphere $ S^3_\varepsilon$ if and only if it is  the curve $\beta(s)$ on the base sphere $S^2(\frac{1}{2})$  is circle of radius $r=\frac{1}{2\sqrt{2-\varepsilon^2}}$ with $\varepsilon^2<1$.
\end{theorem}
\begin{proof}
 Let $\beta : I \to  S^2(\frac{1}{2})$ be an immersed regular curve parameterized by arc length with the geodesic curvature $k_g$.
It follows  from a result in \cite{Ou3} that we can take the horizontal lifts of the tangent and the principal normal
vectors of the curve $\beta: X =a E_1 +b E_2 $ and $\xi= bE_1-a E_2 $ (where $a^2 +b^2=1$) together with $V = E_3$ to be an adapted
orthonormal frame of the Hopf cylinder. A direct computation using (\ref{g3}) gives:

\begin{equation}\label{H1}
\begin{array}{lll}
{\rm Ric}\,(\xi,\xi)=4-2\varepsilon^2,\;
{\rm Ric}\,(\xi,X)=
{\rm Ric}\,(\xi,V)=0.
\end{array}
\end{equation}
We can check that the the geodesic torsion of the lifting curve $\psi^{-1}(\beta(s))$
\begin{equation}\label{H2}
\begin{array}{lll}
\tau_g=-\langle\nabla_X V,\xi\rangle=-\langle\nabla_{a E_1 +b E_2} E_3, bE_1-a E_2\rangle=-\varepsilon.
\end{array}
\end{equation}
It follows from  Eq. (16) in \cite{Ou3} that the surface $\sum$ in $S_{\varepsilon}^3$ is biharmonic if and only if
\begin{equation}\label{Hb3}\notag
\begin{cases}
k''_g-k_g(k^2_g+2\tau_g^2)+k_g{\rm Ric}(\xi,\xi)=0,\\
3k'_gk_g-2k_g{\rm Ric}(\xi,X)=0,\\
 k'_g\tau_g+k{\rm Ric}(\xi,V)=0.
\end{cases}
\end{equation}
Substituting (\ref{H1}) and (\ref{H2}) into the above equation, we get
\begin{equation}\label{H3}
k''_g-k_g\left(k_g^2-(4-4\varepsilon^2)\right)=0,\;
3k'_gk_g=0\;{\rm and}\;
-\varepsilon k'_g=0.
\end{equation}
We solve (\ref{H3})to have $k_g = 0$,  which means that  the surface $\Sigma$ is  minimal surface, or $\beta$ has constant geodesic curvature $k_g^2=4-4\varepsilon^2>0$. It is easy to see from \cite{Ou3} (Page 229) that the mean curvature of the Hopf torus is given by $H = \frac{k_g}{2}$
and $|A|^2 = k_g^2 + 2\tau_g^2 = 4-2\varepsilon^2= constant$. From these we deduce that the Hopf cylinder $\Sigma=\cup_{s\in I}\psi^{-1}(\beta(s))$ is proper
biharmonic if and only if
\begin{equation}\label{H4}
\begin{array}{lll}
H^2=1-\varepsilon^2>0,\;
|A|^2 = 4-2\varepsilon^2>0.
\end{array}
\end{equation}
 We apply the characterizations of Hopf tori in $S^3_\varepsilon$  given in Theorem \ref{Th1} to obtain the Theorem.

\end{proof}
\begin{corollary}\label{coro1}
 A totally umbilical  surface in a Berger 3-sphere  $S^3_\varepsilon$  is proper biharmonic if and only if it is a part of $S^2(1/\sqrt{2})$ in $S^3$.
\end{corollary}
\begin{proof}
It follows  from a result in \cite{Ou4} that a totally umbilical biharmonic surface in 3-dimensional Riemannian manifolds must have constant mean
curvature $H$.  This, together with Theorem \ref{Th1}, implies that the only potential totally umbilical proper
biharmonic surface is a part of $S^2(1/\sqrt{2})$ in $S^3$.
\end{proof}
Since $\varepsilon^2=1$, we see that a potential Berger 3-sphere $S^3_\varepsilon$ has to be 3-sphere $S^3$. Applying Corollary \ref{coro1}, we get
\begin{corollary}\label{coro2}
 A totally umbilical  surface in a Berger 3-sphere  $S^3_\varepsilon$ with $\varepsilon^2\neq1$ is biharmonic if and only if it is minimal.
\end{corollary}

\section{Biharmonic Riemannian submersions from  a Berger 3-sphere $S^3_\varepsilon$}
As Riemannian submersions can be considered as the dual notion  of isometric immersions, it would be interesting to study biharmonic Riemannian submersions. In a recent paper \cite{WO1}, the authors classified all proper Riemannian submersions  from BCV 3-diemnsional spaces into a surface, and proved that  a proper biharmonic Riemannian submersion from a BCV 3-diemnsional space exists only in  $H^2\times\r\to \r^2$,\;or,  $\widetilde{SL}(2,\r)\to \r^2$  of which  some examples were given. In this section, we  give a complete classification of biharmonic Riemannian submersions from a Berger 3-sphere into a surface. \\

Let $\pi:S^3_{\varepsilon} \to (N^2,h)$ be a Riemannian
submersion from  Berger 3-sphere  $S^3_{\varepsilon}$  with an
orthonormal frame  $\{e_1,\; e_2, \;e_3\}$ and  $e_3$ being vertical. Then, we have the following (\ref{R1})-(\ref{GCB1})(see \cite{WO1} for details)
\begin{equation}\label{R1}
\begin{array}{lll}
[e_1,e_3]=f_{3}e_2+\kappa_1e_3,
[e_2,e_3]=-f_{3}e_1+\kappa_2e_3,
[e_1,e_2]=f_1 e_1+f_2e_2-2\sigma e_3,
\end{array}
\end{equation}
where $\{f_1, f_2, f_3,\;\kappa_1,\;\kappa_2,\; \sigma\}
$ is the (generalized) integrability
data of the Riemannian submersion $\pi$.  The frame  $\{e_1,\; e_2, \;e_3\}$ is adapted  to Riemannian
submersion if and only if $f_{3}=0$ holds,  and hence $ \{f_1, f_2,
\kappa_1,\;\kappa_2,\;\sigma\}$ is called the integrability
data  of the adapted frame of the Riemannian submersion $\pi$.  \\
The Levi-Civita connection for the frame  $\{e_1,\; e_2, \;e_3\}$ given by
\begin{equation}\label{R2}
\begin{array}{lll}
\nabla_{e_{1}} e_{1}=-f_1e_2,\;\;\nabla_{e_{1}} e_{2}=f_1
e_1-\sigma e_{3},\;\;\nabla_{e_{1}} e_{3}=\sigma
e_{2},\\\nabla_{e_{2}} e_{1}=-f_2 e_{2}+\sigma
e_3,\;\;\nabla_{e_{2}} e_{2}=f_2 e_{1}, \;\;\nabla_{e_{2}}
e_{3}=-\sigma e_{1},\\ \nabla_{e_{3}}
e_{1}=-\kappa_1e_{3}+(\sigma-f_{3}) e_{2}, \nabla_{e_{3}} e_{2}= -(\sigma-f_{3})
e_{1}-\kappa_2 e_3, \nabla_{e_{3}} e_{3}=\kappa_1 e_{1}+\kappa_2
e_2,
\end{array}
\end{equation}

 the Jacobi identities as
 \begin{equation}\label{Jac}
 \begin{cases}
e_3(f_1)+(\kappa_1+f_2)f_{3}-e_1(f_{3})=0,\\
e_3(f_2)+(\kappa_2-f_1)f_{3}-e_2(f_{3})=0,\\
2 e_3(\sigma)+\kappa_1f_1+\kappa_2f_2+e_2(\kappa_1)-e_1(\kappa_2)=0,
\end{cases}
\end{equation}
and if denoting by $e_i=\sum\limits_{j=1}^3a_{i}^{j}E_j,\;i=1,2,3$,  using (\ref{g1}), (\ref{g2})  and (\ref{R2}), then we have
\begin{equation}\label{RC}
\begin{cases}
R^{M}(e_1,e_3,e_1,e_2)=-e_1(\sigma)+2\kappa_1\sigma=-a_{2}^{3}a_{3}^{3}R,\\
R^{M}(e_1,e_3,e_1,e_3)=e_1(\kappa_1)+\sigma^2-\kappa_{1}^2+\kappa_2f_1=(a_{2}^{3})^2R+\varepsilon^2,\;\\
R^{M}(e_1,e_3,e_2,e_3)=e_1(\kappa_2)-e_3(\sigma)-\kappa_{1}f_{1}-\kappa_1\kappa_2=-a_{1}^{3}a_{2}^{3}R,\;\\
R^{M}(e_1,e_2,e_1,e_2)=e_1(f_2)-e_2(f_1)-f_{1}^{2}-f_{2}^{2}+2f_{3}\sigma-3\sigma^2=(a_{3}^{3})^2R+\varepsilon^2,\\
R^{M}(e_1,e_2,e_2,e_3)=-e_2(\sigma)+2\kappa_2\sigma=a_{1}^{3}a_{3}^{3}R,\\
R^{M}(e_2,e_3,e_1,e_3)=e_2(\kappa_{1})+e_3(\sigma)+\kappa_2 f_2-\kappa_1 \kappa_2=-a_{1}^{3}a_{2}^{3}R,\\
R^{M}(e_2,e_3,e_2,e_3)=\sigma^{2}+e_2(\kappa_2)-\kappa_1f_2- \kappa_2^2=(a_{1}^{3})^2R+\varepsilon^2,
\end{cases}
\end{equation}
where $R=4-4\varepsilon^2$.\\
 Gauss curvature of the base space is given by

\begin{equation}\label{GCB}
K^N=e_1(f_2)-e_2(f_1)-f_1^2-f_2^2+2f_{3}\sigma.
\end{equation}
Clearly,
\begin{equation}\label{GCB0}
e_3(K^N)=e_3\{e_1(f_2)-e_2(f_1)-f_1^2-f_2^2+2f_{3}\sigma\}=0,
\end{equation}
when $f_{3}=0$,  then Gauss curvature of the base space turns into

\begin{equation}\label{GCB1}
K^N=e_1(f_2)-e_2(f_1)-f_1^2-f_2^2.
\end{equation}

We state biharmonic equation for Riemannian submersion from 3-manifolds  which will be later used  in the rest of
this paper.
\begin{lemma}(\cite{WO})\label{Lem1}
Let $\pi:(M^3,g)\to(N^2,h)$ be a Riemannian submersion
with the adapted frame $\{e_1,\; e_2, \;e_3\}$ and the integrability
data $ \{f_1, f_2, \kappa_1,\;\kappa_2,\; \sigma\}$. Then, the
Riemannian submersion $\pi$ is biharmonic if and only if
\begin{equation}\label{lem1}
\begin{cases}
-\Delta^{M}\kappa_1-2\sum\limits_{i=1}^{2}f_i e_i(\kappa_2)-\kappa_2\sum\limits_{i=1}^{2}\left(e_i( f_i)
-\kappa_i f_i\right)+\kappa_1\left(-K^{N}+\sum\limits_{i=1}^{2}f_{i}^{2}\right)
=0,\\
-\Delta^{M}\kappa_2+2\sum\limits_{i=1}^{2}f_i e_i(\kappa_1)+\kappa_1\sum\limits_{i=1}^{2}(e_i( f_i)
-\kappa_i f_i)+\kappa_2\left(-K^{N}+\sum\limits_{i=1}^{2}f_{i}^{2}\right)=0,
\end{cases}
\end{equation}
where
$K^{N}=R^{N}_{1212}\circ\pi=e_1(f_2)-e_2(f_1)-f_{1}^{2}-f_{2}^{2}
$ is  Gauss curvature of Riemannian manifold $(N^2,h)$.
\end{lemma}

\begin{proposition}(see \cite{WO1})\label{L1}
Let $\pi:(M^3,g)\longrightarrow (N^2,h)$ be a Riemannian submersion
from 3-manifolds  with  an orthonormal frame $\{e_1, e_2, e_3\}$ and
$ e_3$ being vertical. If $\nabla_{e_{1}}e_{1}=0$, then either  $\nabla_{e_{2}}e_{2}=0$;\;or, $\nabla_{e_{2}}e_{2}\not\equiv0$,
and the frame $\{e_1, e_2, e_3\}$  is  adapted to the Riemannian submersion $\pi$.
\end{proposition}

  We use Proposition \ref{L1} to  prove the following important consequence  which is used proving our main theorem

\begin{theorem}\label{Cla}
Let $\pi:S_{\varepsilon}^3 \to (N^2,h)$ be a Riemannian submersion
from Berger 3-sphere with  $R=4-4\varepsilon^2\neq0$. Then, we can choose  an adapted  frame  is of the form $\{e_1=a_1^1E_1+a_1^2E_2,\; e_2,\;e_3\}$ to the Riemannian
submersion $\pi$ with $e_3$ being vertical. Moreover, if $E_3$ is not vertical, then $\nabla_{e_2}e_2\neq0$, i.e.,\;$f_2\neq0$.
\end{theorem}

\begin{proof}
It is observed that if  the vertical vector field $ E_3$  is tangent to the fiber of the  Riemannian
submersion $\pi$, then any  basic vector field is of the form $e=a^2E_1+b^2E_2,\;{\rm and}\;a^2+b^2=1$.\\

From this time now, we  just need  to suppose  that the vertical vector field $ e_3 $ is not parallel to $E_3$. Then, the vector filed $e_1=e_3\times E_3$  is  horizontal  and
 hence $\langle e_1, E_3\rangle=0$. From this, a defined orthonormal frame $\{e_1, \;e_2=e_3\times e_1,\;e_3\}$ on $M^3$ is obtained.
If  denoting by $e_i=\sum\limits_{j=1}^{3}a_i^jE_j, i=1,2,3$, together with $\langle e_1, E_3\rangle=0$, then the vector horizontal filed $e_1$  is of the form $e_1=a_1^1E_1+a_1^2E_2$ and hence \;$(a_1^1)^2+(a_1^2)^2=1$. From these, we have the following
\begin{equation}\label{zb}
\begin{array}{lll}
a_1^3=0,\;a_3^{3}\neq\pm1\;{\rm and}\; a_2^{3}\neq0.
\end{array}
\end{equation}
Moreover,  one can also get the following equalities  as
\begin{equation}\label{bb4}
 f_1=0,\;\nabla_{e_1}e_1=0.
\end{equation}
Indeed, we compute
\begin{equation}\label{bb2}
\begin{array}{lll}
\nabla_{e_{1}} e_{1}=\nabla_{e_{1}}(\sum\limits_{i=1}^{3}a_1^iE_i)
=\sum\limits_{i=1}^{3}e_1(a_1^i)E_i+\sum\limits_{i,j=1}^{3}a_1^ja_1^i\nabla_{E_j}E_i.
\end{array}
\end{equation}
In addition, one uses (\ref{R2}) to see that the above has another expression as
\begin{equation}\label{bb1}
\nabla_{e_{1}} e_{1}=-f_1e_{2}=-f_1\sum\limits_{i=1}^{3}a_2^iE_i.
\end{equation}

We equate  (\ref{bb2}) and (\ref{bb1}) and  compare the coefficient of $E_3$ to obatin
 \begin{equation}\label{bb3}
\begin{array}{lll}
-f_1 a_2^3=\langle-f_1\sum\limits_{i=1}^{3}a_1^iE_i,E_3\rangle=\langle\nabla_{e_{1}} e_{1},E_3\rangle\\
=\langle\sum\limits_{i=1}^{3}e_1(a_1^i)E_i+\sum\limits_{i,j=1}^{3}a_1^ja_1^i\nabla_{E_j}E_i,E_3\rangle=e_1(a_1^3)=0,
\end{array}
\end{equation}
which has been used (\ref{g1}) and $a_1^3=0$. This follows  $f_1=0$ for $a_2^3\neq0$, from which we get (\ref{bb4}).\\

 Using (\ref{g1}), (\ref{R2}) and $a_1^3=f_1=0$,  a further calculation analogous  to  those used  computing  (\ref{bb2})--(\ref{bb3}) yields

\begin{equation}\label{thb2}
\begin{cases}
e_1(a_{2}^{3})=-(\sigma+\varepsilon)a_{3}^{3},\\
e_1(a_{3}^{3})=(\sigma+\varepsilon)a_{2}^{3},\\
e_2(a_{2}^{3})=0,\\
e_2(a_{3}^{3})=0,\\
e_3(a_{2}^{3})=-\kappa_2a_{3}^{3},\\
e_3(a_{3}^{3})=\kappa_2a_{2}^{3},\\
\kappa_1a_{3}^{3}=(\sigma-\varepsilon-f_{3})a_{2}^{3},\\
f_2a_{2}^{3}=(\sigma+\varepsilon)a_{3}^{3}.
\end{cases}
\end{equation}

Since $\nabla_{e_{1}}e_{1}=0$,  one concludes from Proposition \ref{L1} to have either $\nabla_{e_{2}}e_{2}\neq0$,  and the
frame $\{e_1, e_2, e_3\}$  is an adapted  to the Riemannian submersion $\pi$, or \;$\nabla_{e_{2}}e_{2}=0$. Now, we just need to consider the latter case
$\nabla_{e_{2}}e_{2}=0$,  i.e., $f_2=0$.  Combining these, one has  the following
\begin{equation}\label{zb1}
a_1^3=f_1=f_2=0,\;a_3^3\neq\pm1\; {\rm and}\;a_2^3\neq0.
\end{equation}

We now show that the above case (i.e., $a_1^3=f_1=f_2=0$, $a_3^3\neq\pm1$,\;$a_2^3\neq0$ ) can not happen by considering the following two cases:\\

Case I: $a_3^3=0$ and $a_1^3=f_1=f_2=0$. One shows that the case can not happen.\\

In this case, since $a_1^3=0$, we have $a_2^3=\pm1$.  We substitute $a_2^3=\pm1$ and $a_3^3=0$  into  the 2nd  equation of (\ref{thb2}) separately to obtain $\sigma=-\varepsilon$ and hence $f_3=\sigma-\varepsilon=-2\varepsilon$. Substituting these and $f_1=f_2=0$ into the 1st and the 2nd equation of (\ref{Jac}), we get $\kappa_1=\kappa_2=0$. From these and using  the 2nd equation of  (\ref{RC}), we get  $a_2^3=0$, a contradiction.\\

Case II: $a_3^3\neq0,\pm1$, $a_2^3\neq0,\pm1$ and $a_1^3=f_1=f_2=0$. We will show that the case can not happen, either.\\
 In this case, substituting $f_2=0$ into the 8th equation of  (\ref{thb2}), we obtain  $\sigma=-\varepsilon$. Then, we apply the 5th equation of  (\ref{RC}), the 1st  and the 2nd equation of (\ref{thb2}) separately to obtain $\kappa_2=0$, $e_1(a_2^3)=e_1(a_3^3)=0$, and hence $e_3(a_2^3)=e_3(a_3^3)=0$ by using the 5th  and the 6th equation of (\ref{thb2}). Combing these and using the 3rd equation and the 4th equation  of (\ref{thb2}), we find $a_2^3, a_3^3$ to be constants. On the other hand, we must have $\kappa_1\neq0$. If otherwise, i.e., $\kappa_1=0$, we then substitute $a_1^3=0=\kappa_1=\kappa_2=0$ and $\sigma=-\varepsilon$ into the 2nd equation of (\ref{RC})
 to $a_2^3=0$ and hence $a_3^3=\pm1$,  a contradiction. Therefore, combining these and using the 1st and the 4th equation of (\ref{RC}), we  deduce  $\kappa_1 $ and $f_3$ being constants. Substituting this and $f_1=f_2=0$ into the 1st equation of the Jacobi identities (\ref{Jac}), we deduce $0=e_1(f_3)=\kappa_1f_3$  meaning $f_3=0$. However, using the 4th equation of (\ref{RC}),  $\sigma=-\varepsilon\neq0$ and $f_1=f_2=f_3=0$, one finds $(a_3^3)^2=\frac{-4\varepsilon^2}{R}=\frac{4\varepsilon^2}{4\varepsilon^2-4}$ implying $4\varepsilon^2-4>0$ and hence $(a_3^3)^2>1$, a contradiction.\\
 Combining Case I and Case II, the case $a_1^3=f_1=f_2=0$ and $a_3^3\neq\pm1$ can not happen.\\  Summarizing all results, we obtain the theorem.

\end{proof}
\begin{remark}\label{r1}
Let $\pi:S^3_{\varepsilon}\to (N^2,h)$ be a Riemannian submersion from a Berger 3-sphere with  $e_3$ being  vertical. Then, we can conclude the following facts:\\

$(a)$: If $R=4-4\varepsilon^2=0$,  then the Berger 3-sphere is
a standard sphere $S^3$.  It is a fact from Theorem 3.3  in \cite{WO}  that biharmonic Riemannian submersion
$\pi:S^3\to (N^2,h)$ has to be harmonic. Actually, the biharmonic Riemannian submersion can be expressed as the Riemannian submersion
$\pi:S^3\to S^2(\frac{1}{2})$, $\psi(z, w) = \frac{1}{2}(2zw, |z|^2 - |w|^2)$, where $z = x^1 + ix^2$, $w = x^3 + ix^4$, up to equivalence. \\

$(b)$: If \;$a_3^3=\pm1$, i.e., the vertical vector field $ e_3$ is parallel to $ E_3$, then it is not difficult to see  from (\ref{g1}) that the tension of the Riemannian  submersion  $\tau(\pi)=- d\pi(\nabla^{M}_{E_3}E_3)=0$, i.e.,\;$\pi$ is harmonic. Moreover, the biharmonic Riemannian submersion can be represented  as the Riemannian submersion $\pi:S^3_{\varepsilon}\to S^2(\frac{1}{2})$, $\psi(z, w) = \frac{1}{2}(2zw, |z|^2 - |w|^2)$, where $z = x^1 + ix^2$, $w = x^3 + ix^4$, up to equivalence.
\end{remark}
\begin{remark}\label{r2}
Let $\pi:S^3_{\varepsilon}\to (N^2,h)$ be a Riemannian submersion from a Berger 3-sphere with  $e_3$ being  vertical. If $ E_3$ is not vertical (i.e.,  \;$a_3^3\neq\pm1$) and $R=4-4\varepsilon^2\neq0$, then it follows  from Theorem \ref{Cla} that there exists such an orthonormal frame $\{e_1=a_1^1E_1+a_1^2E_2, e_2, e_3\}$   adapted to the Riemannian submersion $\pi$, and  $a_1^3=f_1=f_3=0$, $f_2\neq0$ and $a_2^3\neq0$.
\end{remark}

Now, we will prove our main results as follows

\begin{theorem}\label{Ths1}
A Riemannian submersion $\pi:S^3_\varepsilon\to (N^2,h)$ from a Berger 3-sphere to a surface
 is biharmonic if and only if it is harmonic.
\end{theorem}
\begin{proof}
 Let  $\nabla$  denote the Levi-Civita connection on $S^3_{\varepsilon}$  and by $e_i=\sum\limits_{j=1}^3a_i^jE_j,\;i=1,2,3$.
 To obtain the theorem, by Remark \ref {r1}, we just need to consider the case $R=4-4\varepsilon^2\neq0$ and $a_3^{3}\neq\pm1$. Therefore, by Theorem \ref{Cla} and Remark \ref {r2}, we can take an adapted frame as the form $\{e_1=a_1^1E_1+a_1^2E_2,\; e_2,\;e_3\}$  to the Riemannian submersion $\pi$ with $e_3$ being vertical, and together with a result in \cite{WO}, we have
\begin{equation}\label{z1}
a_2^3,\;a_3^3\neq0, \pm1,\;a_1^3=f_1=f_3=0, \;e_3(f_1)=e_3(f_2)=0,\;f_2\neq0.
\end{equation}
We will show that the above case can not happen by the following two steps:\\
 {\bf Step 1}: Show that $e_2(f_2)=e_2(\kappa_1)=e_3(\kappa_1)=e_2(\sigma)=e_3(\sigma)=\kappa_2=0,\;\kappa_1\neq0,\;{\rm and}\;\sigma\neq0,\pm\varepsilon.$\\

 {\bf Firstly}, show that
 $\sigma\neq0$. Clearly, if $\sigma=0$, we use
 the 1st equation of (\ref{RC}) to have  $a_2^3a_3^3=0$ contradicting (\ref{z1}). This leads to $\sigma\neq0$.\\

  {\bf Secondly}, show that $\kappa_2=e_2(\kappa_1)=e_2(\sigma)=e_3(\sigma)=0$.\\
   A straightforward calculation using the 5th , the 6th equation of  (\ref{thb2}), (\ref{GCB0}) and  (\ref{z1}) and applying $e_3$ to both sides of the 4th
 equation of (\ref{RC}) and the 8th equation of  (\ref{thb2}) separately,  we get
\begin{equation}\label{th15}
\begin{array}{lll}
e_3(\sigma)=-\frac{1}{3\sigma}\kappa_2a_2^3a_3^3R,\\
\end{array}
\end{equation}
and
\begin{equation}\label{th16}
\begin{array}{lll}
-f_2\kappa_2a_3^3=e_3(\sigma)a_3^3+(\sigma+\varepsilon)\kappa_2a_2^3.\\
\end{array}
\end{equation}
We substitute (\ref{th15}) into (\ref{th16}) and  simplify
the resulting equation to obtain
\begin{equation}\label{th17}
\begin{array}{lll}
\kappa_2\left(3f_2\sigma a_3^3-a_2^3(a_3^3)^2R+3\sigma(\sigma+\varepsilon)a_2^3\right)=0,\\
\end{array}
\end{equation}
which means $\kappa_2=0$,\;or,
\begin{equation}\label{th19}
\begin{array}{lll}
3f_2\sigma a_3^3-a_2^3(a_3^3)^2R+3\sigma(\sigma+\varepsilon)a_2^3=0.
\end{array}
\end{equation}
Together with (\ref{z1}), substituting the 8th equation of (\ref{thb2})
 into the above equation and  simplifying
the resulting equation,  we have
\begin{equation}\label{th20}
\begin{array}{lll}
3\sigma(\sigma+\varepsilon)=(a_2^3a_3^3)^2R.
\end{array}
\end{equation}
Noting that $f_3=0$,  we rewritten the 8th equation of (\ref{thb2}) as
\begin{equation}\label{f3}
\begin{array}{lll}
\kappa_1a_3^3=(\sigma-\varepsilon)a_2^3.
\end{array}
\end{equation}

Using the 3rd, the 4th equation of (\ref{thb2}) and (\ref{z1}) and applying
$e_2$ to both sides of the above equation gives
\begin{equation}\label{th21}
\begin{array}{lll}
3(2\sigma+\varepsilon)e_2(\sigma)=0.
\end{array}
\end{equation}
This  implies $e_2(\sigma)=0.$\\

From these and using the 5th equation of (\ref{RC}) and (\ref{th15}), we have
$\kappa_2=0$ and hence $e_3(\sigma)=e_2(\kappa_1)=0$.\\

{\bf Thirdly}, show that $e_2(f_2)=e_3(\kappa_1)=0$.\\

Using the 3rd equation of  (\ref{thb2}), (\ref{z1}) and  applying $e_2$ to
both sides of the 2nd equation of (\ref{RC}) and the 8th
equation of (\ref{thb2}) separately, together with $e_2(\kappa_1)=e_2(\sigma)=0$, we get

\begin{equation}\label{th24}
\begin{array}{lll}
e_2e_1(\kappa_1)=0,\;e_2(f_2)=0.
\end{array}
\end{equation}
From these, we get $e_1e_2(\kappa_1)-e_2e_1(\kappa_1)=0$,  which,
together with $e_1e_2(\kappa_1)-e_2e_1(\kappa_1)= [e_1,
e_2](\kappa_1)=-2\sigma e_3(\kappa_1)$,  implies $e_3(\kappa_1)=0$.\\

{\bf Finally}, show that $\sigma\neq\pm\varepsilon$ and $\kappa_1\neq0$. \\
Using (\ref{z1}) and $\kappa_2=0$, the 7th equation
of (\ref{RC}) becomes
\begin{equation}\label{th25}
\begin{array}{lll}
\kappa_1f_2=\sigma^2-\varepsilon^2.
\end{array}
\end{equation}
Since $f_2\neq0$, one sees from  (\ref{th25}) that $\sigma=\pm\varepsilon$ is equivalent to $\kappa_1=0$.  Obviously, if $\sigma=\pm \varepsilon$ and hence $\kappa_1=0$,  by the 2nd equation of (\ref{RC}) and (\ref{z1}), one sees  $R=0$, a contradiction. Therefore, we get $\sigma\neq\pm\varepsilon$ and $\kappa_1\neq0$.\\

 {\bf Step 2}: show that $\sigma =-\varepsilon$, a contradiction.\\
For $f_1=\kappa_2=0$ and $K^N=e_1(f_2)-f_2^2$, \;it is easy to deduce that biharmnic equation (\ref{lem1}) reduces to
\begin{equation}\label{thb26}
\begin{array}{lll}
\Delta\kappa_1-\kappa_1\{-e_1(f_2)+2f_2^2\}=0.
\end{array}
\end{equation}
 Together with (\ref{z1}), using the 1st equation of (\ref{thb2}) and applying $e_1$ to both sides of the 2nd
equation of (\ref{RC}) and (\ref{th25}) separately, we obtain
\begin{equation}\label{th26}
\begin{cases}
e_1e_1(\kappa_1)=2\kappa_1e_1(\kappa_1)-2\sigma e_1(\sigma)-2(\sigma+\varepsilon)a_2^2a_3^3R,\\
\kappa_1e_1(f_2)+f_2e_1(\kappa_1)=2\sigma e_1(\sigma).
\end{cases}
\end{equation}

We substitute the 1st, the 2nd equation of (\ref{RC}),
the 8th equation of (\ref{thb2}), (\ref{th26}), the results of\; Step 1 into biharmnic equation (\ref{thb26}), together with
$f_1=\kappa_2=0$ and (\ref{z1}), to have
\begin{equation}\label{ths27}
\begin{array}{lll}
\kappa_1^3-3\kappa_1^2f_2+\kappa_1(a_2^3)^2R-4(\sigma
+\varepsilon)a_2^3a_3^3R=0.
\end{array}
\end{equation}
 Multiplying  $a_3^3$ to  both  sides of the
above equation  and using the fact that
$\kappa_1a_3^3=(\sigma-\varepsilon)a_2^3$,\;
$\kappa_1f_2=\sigma^2-\varepsilon^2$\; and $(a_2^3)^2+(a_3^3)^2=1$ and  simplifying the resulting
equation we get
\begin{equation}\label{ths28}
\begin{array}{lll}
\kappa_1^2=\frac{5\sigma+3\varepsilon}{\sigma-\varepsilon}R(a_3^3)^2+3\sigma^2-3\varepsilon^2-R.
\end{array}
\end{equation}
We substitute $\kappa_1a_3^3=(\sigma-\varepsilon)a_2^3$  into the
above equation and simplify the resulting equation to obtain
\begin{equation}\label{ths29}
\begin{array}{lll}
\frac{5\sigma+3\varepsilon}{\sigma-\varepsilon}R(a_3^3)^4+(4\sigma^2-2\sigma
\varepsilon-2\varepsilon^2+R)(a_3^3)^2-(\sigma-\varepsilon)^2=0.
\end{array}
\end{equation}
Applying $e_1$ to both sides of (\ref{ths28}) and using the 1st
equation and the 2nd equation of (\ref{RC})  to simplify the
resulting equation we have
\begin{equation}\label{ths30}
\begin{array}{lll}
\kappa_1^3=\kappa_1(7\sigma^2-\varepsilon^2)-\kappa_1(a_2^3)^2R-\frac{8\kappa_1\sigma
\varepsilon R}{(\sigma-\varepsilon)^2}(a_3^3)^2-\frac{4\varepsilon
R^2}{(\sigma-\varepsilon)^2}a_2^3(a_3^3)^3+
\frac{(5\sigma+3\varepsilon)(\sigma+\varepsilon)}{\sigma-\varepsilon}Ra_2^3a_3^3+3\sigma a_2^3a_3^3R.
\end{array}
\end{equation}
We multiply  $a_3^3$ to  both  sides of the above equation and
use the fact that $\kappa_1a_3^3=(\sigma-\varepsilon)a_2^3$  and $(a_2^3)^2+(a_3^3)^2=1$  and
simplify the resulting equation to get
\begin{equation}\label{ths31}
\begin{array}{lll}
\kappa_1^2=-\frac{4\varepsilon
R^2}{(\sigma-\varepsilon)^3}(a_3^3)^4+\frac{(9\sigma^2-5\sigma\varepsilon+
4\varepsilon^2)R}{(\sigma-\varepsilon)^2}(a_3^3)^2+7\sigma^2-\varepsilon^2-R.
\end{array}
\end{equation}
Comparing (\ref{ths28}) with the above equation and
simplifying the resulting equation we have
\begin{equation}\label{ths32}
\begin{array}{lll}
-\frac{4\varepsilon R^2}{(\sigma-\varepsilon)^3}(a_3^3)^4+\frac{(4\sigma^2-3\sigma\varepsilon+
7\varepsilon^2)R}{(\sigma-\varepsilon)^2}(a_3^3)^2+2(2\sigma^2+\varepsilon^2)=0.
\end{array}
\end{equation}
Eq.(\ref{ths32}) multiplied by
$ (5\sigma+3\varepsilon)(\sigma-\varepsilon)^2$ minus Eq.
(\ref{ths29}) multiplied by $(-4\varepsilon R)$, a straightforward  computation yields
\begin{equation}\label{ths33}
\begin{array}{lll}
\left((5\sigma+3\varepsilon)(4\sigma^2-3\sigma\varepsilon+7\varepsilon^2+4\varepsilon(4\sigma^2-2\sigma
\varepsilon-2\varepsilon^2+R)\right)R(a_3^3)^2\\=2(\sigma-\varepsilon)^2\left(-(5\sigma+3\varepsilon)(2\sigma^2+\varepsilon^2)+2\varepsilon R\right).
\end{array}
\end{equation}
On the other hand, Eq.(\ref{ths32})
multiplied by $ (\sigma-\varepsilon)^2$ plus Eq. (\ref{ths29})
multiplied by $2(2\sigma^2+\varepsilon^2)$, and  we then simplify the resulting equation to obtain
\begin{equation}\label{ths34}
\begin{array}{lll}
(\sigma-\varepsilon)\left((4\sigma^2-3\sigma\varepsilon+7\varepsilon^2)R-2(4\sigma^2-2\sigma
\varepsilon-2\varepsilon^2+R)(2\sigma^2+\varepsilon^2)\right)\\=2\left(-(5\sigma+3\varepsilon)(\sigma^2+\varepsilon^2)+2\varepsilon R\right)R(a_3^3)^2.
\end{array}
\end{equation}
A direct computation using Eq. (\ref{ths33}) and Eq.(\ref{ths34}) to
simplify the resulting equation we get
\begin{equation}\label{ths35}
\begin{array}{lll}
\left((5\sigma+3\varepsilon)(4\sigma^2-3\sigma\varepsilon+7\varepsilon^2)+2l(4\sigma^2-2\sigma
\varepsilon-2\varepsilon^2+R)\right)\\
\times \left((4\sigma^2-3\sigma\varepsilon+7\varepsilon^2)R+2(4\sigma^2+2\sigma
\varepsilon-2\varepsilon^2+R)(2\sigma^2+\varepsilon^2)\right)\\
=4(\sigma-\varepsilon)\left(-(5\sigma+3\varepsilon)(2\sigma^2+\varepsilon^2)+2\varepsilon R\right)^2
\end{array}
\end{equation}
A further computation, one finds that the above equation is a polynomial system in $\sigma$ of degree seven
with constant coefficients  as
\begin{equation}\label{ths36}
\begin{array}{lll}
80\sigma^7 +P(\sigma)=0,
\end{array}
\end{equation}
where $P(\sigma)$ denotes a polynomial in $\sigma$ of not more than
6 with constant coefficients. This implies that $\sigma$ have to be constant. Using Eq.(\ref{ths32}) and Eq.(\ref{ths31}), we see that both $a_3^3$  and $\kappa_1$ are constants. Then $a_2^3=\pm\sqrt{1-(a_3^3)^2}$ has to be a constant since $a_1^3=0$. From these and using the 1st equation of
(\ref{thb2}), we obtain $\sigma=-\varepsilon$, which is a
contradiction.\\
Summarizing the results proved above, the theorem follows.
\end{proof}

\end{document}